\documentclass[a4paper,english,dvipsnames]{article}
\usepackage[english]{babel}
\usepackage{graphicx}
\usepackage{parskip}
\usepackage[T1]{fontenc}
\usepackage[utf8]{inputenc}
\usepackage{mathtools}
\usepackage{amssymb}
\usepackage{amsthm}
\theoremstyle{plain}
\begingroup\makeatletter\@for\theoremstyle:=definition,remark,plain\do{\expandafter\g@addto@macro\csname th@\theoremstyle\endcsname{\addtolength{\thm@preskip}{\parskip}}}\endgroup
\allowdisplaybreaks
\usepackage{ltablex}
\keepXColumns
\usepackage{multirow}
\usepackage{tikz}
\usetikzlibrary{tikzmark}
\usepackage{tikz-cd}
\usepackage{emptypage}
\usepackage[nottoc]{tocbibind}
\usepackage[linktoc=all,pdfstartview=FitH,colorlinks=true,linkcolor=Mahogany,citecolor=ForestGreen,urlcolor=MidnightBlue]{hyperref}
\usepackage[capitalise]{cleveref}
\crefname{equation}{}{}
\numberwithin{equation}{section}
\numberwithin{table}{section}
\newtheorem{Thm}{Theorem}[section]

\newtheorem{Def}[Thm]{Definition}
\newtheorem{Exa}[Thm]{Example}
\newtheorem{Lem}[Thm]{Lemma}

\newtheorem{Prop}[Thm]{Proposition}
\newtheorem{Rem}[Thm]{Remark}

\DeclareMathOperator{\diag}{diag}

\DeclareMathOperator{\Sym}{Sym}
\DeclareMathOperator{\Tr}{Tr}
\DeclareMathOperator{\wt}{wt}
\newcommand{\N}{\mathbb{N}}
\newcommand{\R}{\mathbb{R}}
\newcommand{\C}{\mathbb{C}}

\newcommand{\detqQ}{\det\nolimits_{q,Q}}
\newcommand{\glNC}{\mathfrak{gl}(N,\C)}

\newcommand{\hR}{\hat{R}}
\newcommand{\MNC}{M_{N}(\C)}
\newcommand{\Oq}{\mathcal{O}_{q}}
\newcommand{\OqQ}{\mathcal{O}_{q,Q}}

\newcommand{\OqMNC}{\Oq(\MNC)}
\newcommand{\OqQMNC}{\OqQ(\MNC)}

\newcommand{\OqRMNC}{\Oq^{\R}(\MNC)}
\newcommand{\OqQRMNC}{\OqQ^{\R}(\MNC)}

\newcommand{\qbinom}[2]{\genfrac{[}{]}{0pt}{}{#1}{#2}}

\newcommand{\CB}{Cauchy-Binet}
\newcommand{\CH}{Cayley-Hamilton}
\newcommand{\mpq}{multiparameter quantum}
\newcommand{\rea}{reflection equation algebra}
\newcommand{\YBe}{Yang-Baxter equation}
\renewcommand{\hat}[1]{\widehat{#1}}
\renewcommand{\tilde}[1]{\widetilde{#1}}
\renewcommand{\epsilon}{\varepsilon}
\renewcommand{\phi}{\varphi}
\renewcommand{\Phi}{\varPhi}
\renewcommand{\theta}{\vartheta}
\renewcommand{\Theta}{\varTheta}
\newcommand{\mail}[1]{\href{mailto:#1}{\texttt{#1}}}
\newcommand{\textformula}[1]{\text{\parbox{0.8\linewidth}{#1}}}
\newcommand{\figurerelationsarc}{30}
\begin{document}
\title{Multiparameter quantum \CB{} formulas}
\author{Matthias Flor\'e\footnote{Department of Mathematics, Vrije Universiteit Brussel (VUB), B-1050 Brussels, Pleinlaan 2, Belgium, email: \mail{matthias.flore@vub.ac.be}.}}
\date{}
\maketitle
\begin{abstract}
\setlength{\parskip}{0.5\baselineskip}
\setlength{\parindent}{0pt}
\noindent The quantum \CH{} theorem for the generator of the \rea{} has been proven by Pyatov and Saponov, with explicit formulas for the coefficients in the \CH{} formula. However, these formulas do not give an \emph{easy} way to compute these coefficients. Jordan and White provided an elegant formula for the coefficients given with respect to the generators of the \rea{}. In this paper, we provide \CB{} formulas for these coefficients with respect to generators of $\OqQRMNC$, the multiparameter quantized $^{*}$-algebra of functions on $\MNC$ as a real variety, which contains the \rea{} as a subalgebra. We also prove a \CB{} formula for the inverse of a matrix involving these generators.
\end{abstract}
\paragraph{Acknowledgement.} I would like to thank Kenny De Commer for the excellent support and for notifications about relevant literature.
\tableofcontents
\section{Introduction and motivation}
We start with the \mpq{} deformation of $\MNC$, see \cite{ResmqgatqHa}, \cite{SudcmqoGLn}, \cite{ASTqdoGLn} and \cite{DemmqdotgGLn}.

Let $q,q_{ij}\neq 0$, $i,j\in[N]$ be indeterminates, where we write $[N]=\{1,\dots,N\}$ for $N\in\N_{0}$, which all commute, are invertible and satisfy the following conditions:
\begin{align}
q_{ii} & =1 & & \text{and} & q_{ij}q_{ji} & =q^{2}\text{ for all }i\neq j.\label{multiparam}
\end{align}
Hence there are in fact $1+\binom{N}{2}$ indeterminates. In the following, we will work over the fraction field over these indeterminates. We collect the $q_{ij}$ in the matrix $Q=(q_{ij})_{ij}$. We define the $R$-matrix (see also \cref{remarkconventions})
\begin{equation}
R=\sum_{i,j=1}^{N}{q_{ij}e_{ii}\otimes e_{jj}}+(1-q^{2})\sum_{\mathclap{1\leq i<j\leq N}}{e_{ij}\otimes e_{ji}}\in\MNC\otimes\MNC.\label{Rmatrix}
\end{equation}
The $e_{ij}$ are the matrix units. If $q_{ij}=q$, $i\neq j$, we obtain the classical $R$-matrix for quantum $\glNC$, see \cite{Driqg}, \cite[(1.5)]{FRTqoLgaLa} and \cite[Section 8.4.2]{KSqgatr}. In that case, we drop the index $Q$ in the notations. Note that $R=I_{N}\otimes I_{N}$ if $q_{ij}=q=1$, $i\neq j$. The flip operator $\Sigma=\sum_{i,j=1}^{N}{e_{ij}\otimes e_{ji}}\in\MNC\otimes\MNC$ satisfies $\Sigma(v\otimes w)=w\otimes v$ for all $v,w\in\C^{N}$. We write $\hR=\Sigma\circ R$. Then the \YBe{} and the Hecke condition are satisfied:
\begin{gather}
\begin{aligned}
R_{12}R_{13}R_{23} & =R_{23}R_{13}R_{12}, & & \text{or equivalently} & \hR_{12}\hR_{23}\hR_{12} & =\hR_{23}\hR_{12}\hR_{23},
\end{aligned}\label{YB}\\
\begin{aligned}
\hR^{2} & =q^{2}I_{N}\otimes I_{N}+(1-q^{2})\hR, & & \text{or equivalently} & (\hR-I_{N}^{\otimes 2})(\hR+q^{2}I_{N}^{\otimes 2}) & =0.
\end{aligned}\label{Hecke}
\end{gather}
The $R$-matrix is invertible with
\[
R^{-1}=\sum_{i,j=1}^{N}{q_{ij}^{-1}e_{ii}\otimes e_{jj}}+(1-q^{-2})\sum_{\mathclap{1\leq i<j\leq N}}{e_{ij}\otimes e_{ji}}.
\]
We define $\OqQMNC$ as the algebra generated by the matrix entries of $X=X_{01}=\sum_{i,j=1}^{N}{X_{ij}\otimes e_{ij}}$ satisfying the relation
\begin{align}
R_{12}X_{01}X_{02} & =X_{02}X_{01}R_{12}, & & \text{or equivalently} & \hR_{12}X_{01}X_{02}=X_{01}X_{02}\hR_{12}.\label{FRTOqQMNC}
\end{align}
These relations are equivalent to
\begin{align}
q_{ik}X_{ij}X_{kl} & =q_{jl}X_{kl}X_{ij}+\delta_{j>l}(1-q^{2})X_{kj}X_{il}, & i & \geq k.\label{relationsXX}
\end{align}
Hence in $\OqMNC$ we have the relations\\
\begin{minipage}[c]{0.38\textwidth}
\begin{center}
\begin{tikzpicture}[scale=1]
\draw[-](-1,0)--(1,0)node[right]{$q$-comm};
\draw[-](0,-1)--(0,1)node[above]{$q$-comm};
\draw[-](-1,-1)--(1,1)node[right]{comm};
\draw[-](-1,1)--(1,-1)node[right]{$(q-q^{-1})$};
\draw[<->]([shift=({45-\figurerelationsarc}:1)]0,0) arc ({45-\figurerelationsarc}:{45+\figurerelationsarc}:1);
\draw[<->]([shift=({315-\figurerelationsarc}:1)]0,0) arc ({315-\figurerelationsarc}:{315+\figurerelationsarc}:1);
\end{tikzpicture}
\end{center}
\end{minipage}
\hfill\begin{minipage}[c]{0.61\textwidth}
\begin{align*}
X_{ij}X_{ik} & =\left\{\begin{aligned} & qX_{ik}X_{ij} & & \text{if }j<k\\* & q^{-1}X_{ik}X_{ij} & & \text{if }j>k,\end{aligned}\right.\\
X_{ik}X_{jk} & =\left\{\begin{aligned} & qX_{jk}X_{ik} & & \text{if }i<j\\* & q^{-1}X_{jk}X_{ik} & & \text{if }i>j,\end{aligned}\right.\\
[X_{ij},X_{kl}] & =\left\{\begin{aligned} & (q-q^{-1})X_{kj}X_{il} & & \text{if $i<k$, $j<l$}\\* & 0 & & \text{if $i>k$, $j<l$}.\end{aligned}\right.
\end{align*}
\end{minipage}

We define $\OqQRMNC$ by adding to $\OqQMNC$ the matrix entries of $Y=Y_{01}=\sum_{i,j=1}^{N}{Y_{ij}\otimes e_{ij}}$ satisfying the additional relations
\begin{equation}
\left\{\begin{aligned}R_{12}Y_{02}Y_{01} & =Y_{01}Y_{02}R_{12}, & & \text{or equivalently} & \hR_{12}Y_{02}Y_{01} & =Y_{02}Y_{01}\hR_{12},\\*X_{01}R_{12}Y_{02} & =Y_{02}R_{12}X_{01}, & & \text{or equivalently} & X_{02}\hR_{12}Y_{02} & =Y_{01}\hR_{12}X_{01}\end{aligned}\right.\label{FRTOqQRMNC}
\end{equation}
and equipped with the $^{*}$-structure defined by
\begin{align*}
q^{*} & =q, & Q^{*} & =Q, & X^{*} & =Y, & Y^{*} & =X.
\end{align*}
Thus $q_{ij}^{*}=q_{ji}$ and $X_{ij}^{*}=Y_{ji}$. Note that $R_{12}^{*}=R_{21}$, or equivalently $\hR^{*}=\hR$. The $^{*}$ is well defined since applying $^{*}$ to \cref{multiparam} respectively \cref{FRTOqQMNC} results in $q_{ii}=1$, $q_{ij}q_{ji}=q^{2}$, $i\neq j$ respectively $Y_{02}Y_{01}\hR_{12}=\hR_{12}Y_{02}Y_{01}$, which is precisely \cref{multiparam} respectively the first set of relations in \cref{FRTOqQRMNC} and analogously for the other sets of relations in \cref{FRTOqQRMNC}. The second set of relations in \cref{FRTOqQRMNC} is equivalent to
\begin{equation}
q_{jk}X_{ij}Y_{kl}+\delta_{jk}(1-q^{2})\sum_{m=1}^{j-1}{X_{im}Y_{ml}}=q_{il}Y_{kl}X_{ij}+\delta_{il}(1-q^{2})\sum_{n=i+1}^{N}{Y_{kn}X_{nj}},\label{relationsXY}
\end{equation}
$i,j,k,l\in[N]$. If $q_{ij}=q=1$, $i\neq j$ then all these relations give commutativity and define the coordinate algebra of regular functions on $\MNC$ as a real variety.

Using \cref{FRTOqQMNC} and \cref{FRTOqQRMNC}, we verify that $A=XY$ satisfies the \emph{reflection equation}
\begin{equation}
A_{01}\hR_{12}A_{01}\hR_{12}=\hR_{12}A_{01}\hR_{12}A_{01},\label{refleq}
\end{equation}
or equivalently $R_{12}A_{01}R_{21}A_{02}=A_{02}R_{12}A_{01}R_{21}$. This is the reflection equation used in \cite{JWtcotreavqm}, for $q_{ij}=q$, $i\neq j$ and with $q$ interchanged with $q^{-1}$. The \emph{\rea{}} is defined as the algebra generated by the $N^{2}$ entries of the matrix $A$. One can verify that \cref{refleq} are the universal relations for this subalgebra of $\OqQRMNC$. Similarly, $B=YX$ satisfies
\[
B_{02}\hR_{12}B_{02}\hR_{12}=\hR_{12}B_{02}\hR_{12}B_{02}.
\]
In e.g.~\cite{GPSHsacrorea} the \emph{\CH{} theorem} is proved for the generating matrix $A=XY$ of the \rea{}:
\begin{equation}
\sum_{i=0}^{N}{\sigma_{q,Q}(i)(-XY)^{N-i}}=0.\label{CH}
\end{equation}
The $\sigma_{q,Q}(i)$, $0\leq i\leq N$ are self-adjoint and central in $\OqQRMNC$ (see \cref{selfadjointandcentral}). Hence they are also central in the \rea{} and in fact, they generate the center of the \rea{}, see \cite{FRTqoLgaLa} and \cite{GPSHsacrorea}. We will prove that they also satisfy an invariance property, see \cref{invariance}.

The formulas \cite[(3.2) and (3.3)]{GPSHsacrorea} (\cref{sigmaqQiGPS} and \cref{Newton}) for the coefficients $\sigma_{q,Q}(i)$ in the \CH{} theorem \cref{CH} do not give a \emph{precise} description of them. In \cite[Theorem 1.3]{JWtcotreavqm} (\cref{sigmaqiJW}), an elegant formula for the coefficients in the case that $q_{ij}=q$, $i\neq j$ is given with respect to the generators of the \rea{}. In particular:
\begin{equation}
\textformula{It is a $q$-deformation of the commutative case where all coefficients are integral powers of $q$ (a priori it is possible to have a term with e.g.~the factor $q^{2}+q^{2}(q-q^{-1})$). In particular, there are no terms which vanish if $q\to 1$ (a priori it is possible to have a term with e.g.~the factor $q-q^{-1}$).}\label{observation}
\end{equation}
With respect to the decomposition of $\OqQRMNC$ in its $X$ and $Y$-constituents, it is however not clear from \cref{sigmaqQiGPS} and \cref{Newton} whether the $\sigma_{q,Q}(i)$ satisfy \cref{observation} for a specific ordering. Note that the formula for the coefficients in \cite[Theorem 1.3]{JWtcotreavqm} (\cref{sigmaqiJW}) satisfies \cite[Item 1 and 2 on page 100]{GKLLRTnsf}, which are indicated to hold for an interesting class of noncommutative algebras. Also note that \cite[Theorem 1.3]{JWtcotreavqm} (\cref{sigmaqiJW}) is a quantum analogue of the commutative case, where for all $N\geq 2$ and $0\leq i\leq N$, $\sigma(i)$ is the sum of the $i\times i$-principal minors of  $A$, thus
\begin{equation}
\sigma(i)=\sum_{J\in\binom{[N]}{i}}{[A]_{J,J}}.\label{sigmaicomm}
\end{equation}
Here, we write $\binom{[N]}{i}$ for the set of $i$-element subsets of $[N]$, thus $\binom{[N]}{i}=\{J=(j_{1},\dots,j_{i})\mid 1\leq j_{1}<\dots<j_{i}\leq N\}$ for $i\in[N]$ and $\binom{[N]}{0}=\{\emptyset\}$. For $J,K\in\binom{[N]}{i}$ we denote by $[A]_{J,K}$ the $i\times i$-minor of the $N\times N$-matrix $A$ with rows $j_{1},\dots,j_{i}$ and columns $k_{1},\dots,k_{i}$.

In \cref{theoremtINplusXYinverse} we compute $(tI_{N}+XY)^{-1}$, motivated by the observation that for $N=2$ and $T$ (see \cref{exampleTtwothree}), \cite[(4.14)]{WorCstaragbue} (see \cref{ItwoplusTstarTinverse}) satisfies \cref{observation}. In the commutative case, for all $N\geq 2$, the entries of $(tI_{N}+A)^{-1}$ are given by
\begin{equation}
((tI_{N}+A)^{-1})_{ij}=\frac{1}{\mathcal{C}_{N}}\sum_{\substack{K\subset[N]\\i,j\in K}}{t^{N-|K|}(-1)^{|K_{>i}|+|K_{>j}|}[A]_{K\setminus\{j\},K\setminus\{i\}}}\label{tINplusXYinversecomm}
\end{equation}
where $\mathcal{C}_{N}=\sum_{k=0}^{N}{\sigma(k)t^{N-k}}$ and $K_{>i}=\{k\in K\mid k>i\}$. If $A=XY$ then we can use the classical \CB{} formula
\[
[A]_{I,J}=\sum_{K\in\binom{[N]}{i}}{[X]_{I,K}[Y]_{K,J}}
\]
to rewrite \cref{sigmaicomm} and \cref{tINplusXYinversecomm} in terms of minors of $X$ and $Y$. In \cref{theoremsigmaqQiCB} and \cref{theoremtINplusXYinverse}, we give \mpq{} analogues of these combined formulas, which can thus be seen as \emph{\mpq{} \CB{} formulas}. These will be proved later in this paper. We will use the notation
\begin{align}
D & =\diag(1,q,\dots,q^{N-1}), & D' & =\diag(q^{N-1},\dots,q,1).\label{DandDprime}
\end{align}
The \mpq{} minors $[\cdot]_{q,Q,J,K}$ are given in \cref{mpqminor} and \cref{mpqminorofXandYwithDandDprime}. For $J\subset[N]$, $k\in[N]$ and $\diamond$ being $<$ or $>$, we will write
\begin{align*}
(-q)_{J\diamond k} & =\prod_{\substack{j\in J\\j\diamond k}}{(-q_{jk})}, & (-q)_{k\diamond J} & =\prod_{\substack{j\in J\\k\diamond j}}{(-q_{kj})}, & q_{Jk} & =\prod_{j\in J}{q_{jk}}.
\end{align*}
\begin{Thm}\label{theoremsigmaqQiCB} For all $N\geq 2$ and $0\leq i\leq N$ we have
\[
\sigma_{q,Q}(i)=\sum_{\mathclap{J,K\in\binom{[N]}{i}}}{[DX]_{q,Q,J,K}[YD]_{q,Q,K,J}}=\sum_{\mathclap{J,K\in\binom{[N]}{i}}}{[D'Y]_{q,Q,J,K}[XD']_{q,Q,K,J}}.
\]
In particular, $\sigma_{q,Q}(i)$ satisfies \cref{observation} if the factors are arranged as above.
\end{Thm}
\begin{Thm}\label{theoremtINplusXYinverse} Let $t$ be a central indeterminate element in $\OqQRMNC$, e.g.~a parameter. For all $N\geq 2$, the entries of $(tI_{N}+XY)^{-1}$ are given by
\begin{align*}
 & ((tI_{N}+XY)^{-1})_{ij}\\*
={} & \frac{1}{\mathcal{C}_{N}}\sum_{\substack{K\subset[N]\\i,j\in K}}{t^{N-|K|}(-q)_{i<K}(-q)_{K>j}\sum_{L\in\binom{[N]}{|K|-1}}{[D'Y]_{q,Q,L,K\setminus\{i\}}[XD']_{q,Q,K\setminus\{j\},L}}}
\end{align*}
where we formally invert $\mathcal{C}_{N}=\sum_{k=0}^{N}{\sigma_{q,Q}(k)t^{N-k}}$, which is a central element. In particular, $(tI_{N}+XY)^{-1}$ satisfies \cref{observation} if the factors are arranged as above.
\end{Thm}
\begin{Rem}\label{remarkconventions} In comparison to \cite{FRTqoLgaLa}, \cite{PWqlg}, \cite{PScrfqm}, \cite{GPSHsacrorea}, \cite{KSqgatr}, \cite{ZhatqCHt} and \cite{JWtcotreavqm}, we interchange $q$ and $q^{-1}$ because then we have positive powers of $q$ in \cref{Rmatrix}, \cref{Hecke}, \cref{relationsXX}, \cref{relationsXY}, \cref{DandDprime}, \cref{theoremsigmaqQiCB}, \cref{theoremtINplusXYinverse}, \cref{qnumber}, \cref{RLC}, \cref{normLC}, \cref{sigmaqTN}, \cref{tableTtwothree}, \cref{sigmaqQiseparated}, \cref{mpqminorofXandYwithDandDprime}, \cref{sigmaqQN} and \cref{ItwoplusTstarTinverse}. We also have an extra factor $q$ by $R$ in \cref{Rmatrix} because then the factor $q$ in \cite[(3.1) and also (2.24) after our next remark on $q$-numbers]{GPSHsacrorea} disappears. This also causes an extra power of $q$ in \cref{Hecke}, \cref{RLC} and \cref{sigmaqQiGPS}. Lastly, we also use another convention of $q$-numbers in \cref{qnumber} because then abundant powers of $q$ in \cref{sqQzero}, \cref{normLC}, \cref{alphai} and \cref{Newton} disappear.
\end{Rem}
\section{The \CH{} theorem in the \rea{}}
To give the coefficients $\sigma_{q,Q}(i)$ in \cref{CH}, we recall some properties and notations from \cite{GPSHsacrorea}. First we recall (\cite[Chapter 2]{KSqgatr}) the $q$-numbers. We write $n_{q}=\frac{1-q^{n}}{1-q}=1+q+q^{2}+\dots+q^{n-1}$ (allowing $q=1$). We will use the \emph{$q$-number}
\begin{equation}
[n]=n_{q^{2}}=\frac{1-q^{2n}}{1-q^{2}}.\label{qnumber}
\end{equation}
The \emph{$q$-factorial} is $[n]!=[1][2]\dots[n]$ if $n\in\N_{0}$, $[0]!=1$ and the \emph{$q$-binomial coefficients} are $\qbinom{N}{i}=\frac{[N]!}{[i]![N-i]!}$. Note that all these expressions converge to their classical counterpart if $q\to 1$ since $[n]\to n$ if $q\to 1$.

In order to apply \cite{GPSHsacrorea}, we need that $R_{12}^{t_{1}}$ is invertible, where $t_{1}$ is the transpose in the first leg, thus
\[
R_{12}^{t_{1}}=\sum_{i,j=1}^{N}{q_{ij}e_{ii}\otimes e_{jj}}+(1-q^{2})\sum_{\mathclap{1\leq i<j\leq N}}{e_{ji}\otimes e_{ji}}.
\]
This condition is satisfied since
\[
(R_{12}^{t_{1}})^{-1}=\sum_{i,j=1}^{N}{q_{ij}^{-1}e_{ii}\otimes e_{jj}}+(1-q^{-2})\sum_{\mathclap{1\leq i<j\leq N}}{q^{2(j-i)}e_{ji}\otimes e_{ji}},
\]
which is easily verified.
In \cite{GPSHsacrorea}, the matrices $\mathcal{C}$ and $\mathcal{B}$ are defined as
\begin{align*}
\mathcal{C} & =\Tr_{(1)}(((R_{12}^{t_{1}})^{-1})_{12}^{t_{1}}\Sigma_{12}), & \mathcal{B} & =\Tr_{(2)}(((R_{12}^{t_{1}})^{-1})_{12}^{t_{1}}\Sigma_{12}),
\end{align*}
where $\Tr_{(i)}$ is the trace in the $i$-th leg. We compute that $\mathcal{C}=D^{2}$ and $\mathcal{B}=(D')^{2}$. In \cite{GPSHsacrorea}, the \emph{quantum power sum symmetric polynomials} are defined as
\begin{align}
s_{q,Q}(k) & =\Tr(\mathcal{C}A^{k})=\Tr(D^{2}(XY)^{k}), & k & \geq 0.\label{sqQk}
\end{align}
It is clear that the $s_{q,Q}(k)$, $k\geq 0$ are self-adjoint. Note that
\begin{equation}
s_{q,Q}(0)=\Tr(D^{2})=\Tr((D')^{2})=[N].\label{sqQzero}
\end{equation}
For a function $\sigma:I\to\N$ with $I$ a finite subset of $\N$, we define the \emph{length} by
\[
l(\sigma)=|\{(i,j)\in I^{2}\mid i<j\text{ and }\sigma(i)>\sigma(j)\}|.
\]
We also define
\begin{align}
(-q)_{\sigma,r} & =\prod_{\substack{i,j\in I,i<j\\\sigma(i)>\sigma(j)}}{(-q_{\sigma(i)\sigma(j)})}, & (-q)_{\sigma,c} & =\prod_{\substack{i,j\in I,i<j\\\sigma(i)>\sigma(j)}}{(-q_{\sigma(j)\sigma(i)})}.\label{minusqsigmarandminusqsigmac}
\end{align}
The notation $r$ (row) and $c$ (column) is motivated by the formula for the \mpq{} minors in \cref{mpqminor}, where $(-q)_{\sigma,r}$ respectively $(-q)_{\sigma,c}$ belongs to the permutation which occurs in the row respectively column index. Note that $(-q)_{\sigma,r}^{*}=(-q)_{\sigma,c}$. The \emph{\mpq{} Levi-Civita tensor} is given by
\begin{equation}
u_{1\dots N}=\sum_{\sigma\in S_{N}}{(-q)_{\sigma,c}e_{\sigma(1)}\otimes\dots\otimes e_{\sigma(N)}}\in\bigotimes_{k=1}^{N}{\C^{N}},\label{LC}
\end{equation}
where $e_{i}$ is the $N\times 1$-column vector with $1$ on position $(i,1)$ and $0$ elsewhere and $S_{N}$ is the symmetric group of permutations of $[N]$. We verify that
\begin{align}
\hR_{i,i+1}u_{1\dots N} & =-q^{2}u_{1\dots N}, & 1 & \leq i\leq N-1.\label{RLC}
\end{align}
Its $2$-norm squared satisfies
\begin{equation}
\|u_{1\dots N}\|^{2}=u_{1\dots N}^{*}u_{1\dots N}=\sum_{\sigma\in S_{N}}{q^{2l(\sigma)}}=[N]!\label{normLC}
\end{equation}
and we define $v_{1\dots N}=u_{1\dots N}^{*}$. By applying $^{*}$ to \cref{RLC}, we obtain $v_{1\dots N}\hR_{i,i+1}=-q^{2}v_{1\dots N}$, $1\leq i\leq N-1$. Using \cite[(3.2)]{GPSHsacrorea}, the coefficients in \cref{CH}, which are the \emph{quantum elementary symmetric polynomials}, are given by $\sigma_{q,Q}(0)=1$ and
\begin{equation}
\sigma_{q,Q}(i)=q^{-(i-1)i}\alpha_{i}v_{1\dots N}((XY)_{01}\hR_{12}\dots\hR_{i-1,i})^{i}u_{1\dots N}\label{sigmaqQiGPS}
\end{equation}
for $i\in[N]$, where
\begin{equation}
\alpha_{i}=\frac{1}{\|u_{1\dots N}\|^{2}}\qbinom{N}{i}.\label{alphai}
\end{equation}
By \cite[Proposition 1]{GPSHsacrorea}, the $\sigma_{q,Q}(i)$ and $s_{q,Q}(i)$ are related by the \emph{Newton relations}
\begin{align}
[k]\sigma_{q,Q}(k)+\sum_{i=1}^{k}{(-1)^{i}s_{q,Q}(i)\sigma_{q,Q}(k-i)} & =0, & 1 & \leq k\leq N.\label{Newton}
\end{align}
\begin{Prop}\label{tracecommutation} If the $N\times N$-matrices $M$, $P$ in noncommuting variables satisfy the relation $R_{12}P_{01}M_{02}=M_{02}R_{12}P_{01}$ then $[\Tr(D^{2}M),P_{ij}]=0$ for all $i,j\in[N]$.
\end{Prop}
\begin{proof} We write $R_{12}P_{01}M_{02}$ and $M_{02}R_{12}P_{01}$ out. Hence we obtain sets of relations, from which the result follows.
\end{proof}
\begin{Lem} For all $k\geq 0$ we have
\begin{equation}
R_{12}X_{01}((XY)^{k})_{02}=((XY)^{k})_{02}R_{12}X_{01}.\label{FRTforhigherordercommutation}
\end{equation}
\end{Lem}
\begin{proof} This follows from \cref{FRTOqQMNC}, \cref{FRTOqQRMNC} and induction.
\end{proof}
\begin{Thm}\label{selfadjointandcentral} The $\sigma_{q,Q}(i)$, $0\leq i\leq N$ and the $s_{q,Q}(k)$, $k\geq 0$ are self-adjoint and central in $\OqQRMNC$ (see also \cite[Lemma 3.41]{DCFtfoqGLNC}).
\end{Thm}
\begin{proof} From \cref{tracecommutation} and \cref{FRTforhigherordercommutation}, it follows that the $s_{q,Q}(k)$ commute with all $X_{ij}$. Since the $s_{q,Q}(k)$ are self-adjoint by \cref{sqQk}, it follows that they also commute with all $X_{ij}^{*}$. For the $\sigma_{q,Q}(i)$, the result now follows from \cref{Newton}.
\end{proof}
In \cite{JWtcotreavqm}, the $\sigma_{q}(i)$ are described in terms of the entries $a_{ij}$ of the matrix $A=XY$. In order to state this result, we use the same notation as in \cite{JWtcotreavqm}. For $J\in\binom{[N]}{i}$ we write $\Sym(J)$ for the subgroup of $S_{N}$ consisting of permutations which fix the complement of $J$ in $[N]$ and we have the \emph{weight} $\wt(J)=\sum_{l=1}^{i}{j_{l}}$. For $\sigma\in S_{N}$ we have its \emph{exceedance} $e(\sigma)=|\{i\in[N]\mid\sigma(i)>i\}|$. Now \cite[Theorem 1.3]{JWtcotreavqm} (with $\sigma_{q}(k)=q^{2k}c_{k}$ by \cite[(1.4)]{JWtcotreavqm} and $q$ interchanged with $q^{-1}$, see \cref{remarkconventions}) states that
\begin{equation}
\sigma_{q}(i)=\sum_{J\in\binom{[N]}{i}}{\sum_{\sigma\in\Sym(J)}{q^{2(\wt(J)-i)}(-q)^{-l(\sigma)}q^{-e(\sigma)}a_{j_{1}\sigma(j_{1})}\dots a_{j_{i}\sigma(j_{i})}}}.\label{sigmaqiJW}
\end{equation}
In \cref{theoremsigmaqQiCB}, we describe the $\sigma_{q,Q}(i)$ inside $\OqQRMNC$ in terms of \mpq{} minors of $X$ and $Y$ and in \cref{invariance}, we prove an invariance property for the $\sigma_{q,Q}(i)$.
\begin{Exa}\label{exampleTtwothree} To simplify the calculations in this Example, we introduce the quotient $^{*}$-algebra of $\OqRMNC$ by the relations $X_{ij}=0$ for $1\leq j<i\leq N$ and $X_{ii}^{*}=X_{ii}$ for $i\in[N]$. We denote by $\pi_{T}$ the corresponding quotient map. In this case we write the generating matrices as $X=T$ and $Y=T^{*}$. Note that \cref{CH} also holds for $TT^{*}$ by applying $\pi_{T}$. In this case we write the coefficients as $\sigma_{q,T}(i)=\pi_{T}(\sigma_{q}(i))$. Moreover, $\pi_{T}$ is faithful on the \rea{}. In \cref{theoremsigmaqQiCB}, we describe the $\sigma_{q,Q}(i)$ where factors $X_{ij}$ are followed by factors $Y_{ij}$ or vice versa. Since $T_{ii}^{*}=T_{ii}$ for $i\in[N]$, it is natural to place the $T_{ii}$ in the middle. This is possible since $[T_{ii},T_{jk}]=0$ if $i\neq j$ and $i\neq k$. For $i=1$, it follows from \cref{Newton} that $\sigma_{q,Q}(1)=s_{q,Q}(1)$. For $i=N$ and $T$, we have in \cite{PScrfqm} that
\begin{equation}
\sigma_{q,T}(N)=q^{(N-1)N}T_{11}^{2}\dots T_{NN}^{2}.\label{sigmaqTN}
\end{equation}
These give all $\sigma_{q,T}(i)$ for $N=2$, $3$ in the first column of \cref{tableTtwothree} except for $\sigma_{q,T}(2)$ for $N=3$, which can be computed using e.g.~\cref{sigmaqQiGPS}, \cref{sigmaqiJW}, \cref{sigmaqQiseparated} or immediately from \cref{theoremsigmaqQiCB}. The second column can then be computed using the relations for the $T_{ij}$, or immediately from \cref{invariance}. The operation $'$ is defined in \cref{definitionprime}.
\vspace{8pt}
\begin{tabularx}{\linewidth}{@{}r*{2}{@{\hspace{0.4em}}c@{\hspace{0.4em}}>{\raggedright\arraybackslash}X}@{}}
 & & \multicolumn{1}{@{}c@{}}{$T_{ij}$ followed by $T_{ij}^{*}$\tikzmark{markstarback}} & & \multicolumn{1}{@{}c@{}}{\tikzmark{markstarfront}$T_{ij}^{*}$ followed by $T_{ij}$}\begin{tikzpicture}[overlay,remember picture]\draw[yshift=1em,->]({pic cs:markstarback})[bend left=20] to node[below]{$'$} ({pic cs:markstarfront});\draw[->]({pic cs:markstarfront})[bend left=20] to node[above]{$'$} ({pic cs:markstarback});\end{tikzpicture}\\
 & & \multicolumn{1}{@{}c@{}}{\downbracefill} & & \multicolumn{1}{@{}c@{}}{\downbracefill}\\*
\multicolumn{2}{@{}l@{}}{$N=2$} & \multicolumn{1}{@{}c@{}}{$D=\diag(1,q)$} & & \multicolumn{1}{@{}c@{}}{$D'=\diag(q,1)$}\\\hline
$\sigma_{q,T}(0)$ & $=$ & $1$ & $=$ & $1$\\
$\sigma_{q,T}(1)$ & $=$ & $T_{11}^{2}+q^{2}T_{22}^{2}+T_{12}T_{12}^{*}$ & $=$ & $q^{2}T_{11}^{2}+T_{22}^{2}+T_{12}^{*}T_{12}$\\
$\sigma_{q,T}(2)$ & $=$ & $q^{2}T_{11}^{2}T_{22}^{2}$ & $=$ & $q^{2}T_{11}^{2}T_{22}^{2}$\\
\\
\multicolumn{2}{@{}l@{}}{$N=3$} & \multicolumn{1}{@{}c@{}}{$D=\diag(1,q,q^{2})$} & & \multicolumn{1}{@{}c@{}}{$D'=\diag(q^{2},q,1)$}\\\hline
$\sigma_{q,T}(0)$ & $=$ & $1$ & $=$ & $1$\\
$\sigma_{q,T}(1)$ & $=$ & $T_{11}^{2}+q^{2}T_{22}^{2}+q^{4}T_{33}^{2}+T_{12}T_{12}^{*}+T_{13}T_{13}^{*}+q^{2}T_{23}T_{23}^{*}$ & $=$ & $q^{4}T_{11}^{2}+q^{2}T_{22}^{2}+T_{33}^{2}+q^{2}T_{12}^{*}T_{12}+T_{13}^{*}T_{13}+T_{23}^{*}T_{23}$\\
$\sigma_{q,T}(2)$ & $=$ & $q^{2}T_{11}^{2}T_{22}^{2}+q^{4}T_{11}^{2}T_{33}^{2}+q^{6}T_{22}^{2}T_{33}^{2}+q^{2}T_{23}T_{11}^{2}T_{23}^{*}+q^{2}T_{13}T_{22}^{2}T_{13}^{*}+q^{4}T_{12}T_{33}^{2}T_{12}^{*}+q^{2}T_{12}T_{23}T_{12}^{*}T_{23}^{*}-qT_{12}T_{23}T_{22}T_{13}^{*}-q^{3}T_{13}T_{22}T_{12}^{*}T_{23}^{*}$ & $=$ & $q^{6}T_{11}^{2}T_{22}^{2}+q^{4}T_{11}^{2}T_{33}^{2}+q^{2}T_{22}^{2}T_{33}^{2}+q^{4}T_{23}^{*}T_{11}^{2}T_{23}+q^{2}T_{13}^{*}T_{22}^{2}T_{13}+q^{2}T_{12}^{*}T_{33}^{2}T_{12}+q^{2}T_{12}^{*}T_{23}^{*}T_{12}T_{23}-qT_{13}^{*}T_{22}T_{12}T_{23}-q^{3}T_{12}^{*}T_{23}^{*}T_{22}T_{13}$\\
$\sigma_{q,T}(3)$ & $=$ & $q^{6}T_{11}^{2}T_{22}^{2}T_{33}^{2}$ & $=$ & $q^{6}T_{11}^{2}T_{22}^{2}T_{33}^{2}$\\*
\\*\caption{$\sigma_{q,T}(i)$ for $N=2$, $3$.}\label{tableTtwothree}
\end{tabularx}
In particular, $\sigma_{q,T}(i)$ for $N=2$, $3$ and $0\leq i\leq N$ satisfy \cref{observation} if the factors are arranged as in \cref{tableTtwothree}. The usage of $D$ and $D'$ comes from \cref{theoremsigmaqQiCB} and \cref{mpqminorofXandYwithDandDprime}.
\end{Exa}
\begin{Rem} In $D'$, we have the numbers $N-1,\dots,1,0$ in the exponents on the diagonal and in \cref{definitionprime}, the transpose $\rho$ over the anti-diagonal is involved. The matrix $C=\diag(N-1,\dots,1,0)$ and $\rho$ also appear in the Capelli identity, which states (see e.g.~\cite{NWanotCifspoHt}) that for $N\times N$-matrices $x$, where $x_{ij}$ are commuting variables, $\partial$, with $\partial_{ij}=\frac{\partial}{\partial x_{ij}}$, $E=x^{T}\partial$ and $F=\rho(x)^{T}\rho(\partial)$, we have
\[
\det(E+C)=\det(x)\det(\partial)=\det(\rho(F)+\rho(C)).
\]
Here, the column-determinant is used at the left and right hand side (we need to specify the ordering since the entries of $E+C$ and $\rho(F)+\rho(C)$ do not commute). The matrices $C$ and $\rho(C)$ appear in more formulas involving noncommutative matrices: the Turnbull and Howe-Umeda-Kostant-Sahi identities (see e.g.~\cite{FZcpoCaTifcit}), \cite[(1.9), (1.11), (1.13), (1.15), (1.16), (1.17), (1.19), (1.26), (1.28), (1.34), (1.36), (3.26), (3.28)]{CSSndCBfaCti} and \cite[Theorem 6, page 39]{CFRapoMm}.

In \cite{FZcpoCaTifcit}, we have $\det(x^{T}p+hC)$ with $p=h\partial$, where the commutative case corresponds to $h=0$. If we set $q=e^{\frac{h}{2}}$ then $D'=e^{\frac{h}{2}C}$ and we have the analogy between $\frac{h}{2}C+x^{T}p+\frac{h}{2}C$ and $e^{\frac{h}{2}C}X^{*}Xe^{\frac{h}{2}C}$ in \cref{theoremsigmaqQiCB} and \cref{invariance}.

In \cite[(3.2.3)]{NUWaqaotCi}, a quantum analogue of the Capelli identity is given, which is generalized to the \mpq{} case in \cite[Theorem 4.6]{JZCiomqlg}.

All this suggests a connection between these results, although there does not seem to be a common framework at first sight.
\end{Rem}
\section{Invariance of the coefficients}
The aim of this Section is to prove the invariance property in \cref{invariance}. Although this invariance property is highly non-trivial (e.g.~for $N=3$ already it is a long computation to verify that $\sigma_{q,Q}(2)=(\sigma_{q,Q}(2))'$ and if $N$ increases, the computations get more complex), we will give a short and elementary proof of this invariance.
\begin{Def}\label{definitionprime} We define $':[N]\to[N]:i\mapsto i'=N-i+1$. For $J\subset[N]$ we write $J'=N-J+1=\{j'\in[N]\mid j\in J\}$. We define $':\OqQRMNC\to\OqQRMNC$ as the unique $^{*}$-anti-isomorphism such that
\begin{align*}
q' & =q, & q_{ij}' & =q_{i'j'}, & X_{ij}' & =X_{j'i'}, & Y_{ij}' & =Y_{j'i'}.
\end{align*}
More generally, we extend $'$ to $n\times m$-matrices $A$ over $\OqQRMNC$ by
\[
(A')_{ij}=a_{n-j+1,m-i+1}'.
\]
Thus $'$ acts as follows:
\begin{enumerate}
\item Reverse the ordering of the factors in each term in each entry.
\item Substitute $q_{ij}$ by $q_{i'j'}$, $X_{ij}$ by $X_{j'i'}$ and $Y_{ij}$ by $Y_{j'i'}$. This means that we substitute $X_{ij}$ respectively $Y_{ij}$ by its image under transposing the matrix $X$ respectively $Y$ over the anti-diagonal.
\item Transpose the matrix over its anti-diagonal.
\end{enumerate}
\end{Def}
For example, $(qX_{12}X_{23}X_{22}X_{13}^{*})'=qX_{13}^{*}X_{22}X_{12}X_{23}$ for $N=3$. Note that this is consistent with \cref{DandDprime}, $'$ is an involution, $'\circ\pi_{T}=\pi_{T}\circ{}'$,
\begin{align*}
Q' & =Q^{T}, & X' & =X, & Y' & =Y, & R_{12}' & =R_{12}.
\end{align*}
The restriction of $'$ to $1\times 1$-matrices over $\OqMNC$ gives $\rho_{q}$ in \cite[Proposition 3.7.1]{PWqlg}.
\begin{Lem} The operation $'$ is well defined.
\end{Lem}
\begin{proof} Applying $'$ to \cref{multiparam} respectively \cref{FRTOqQMNC} results in $q_{i'i'}=1$, $q_{i'j'}q_{j'i'}=q^{2}$, $i\neq j$ respectively $X_{02}X_{01}R_{12}=R_{12}X_{01}X_{02}$, which is precisely \cref{multiparam} respectively \cref{FRTOqQMNC} and analogously for the sets of relations in \cref{FRTOqQRMNC}.
\end{proof}
\begin{Lem} For an $n\times m$-matrix $A$ and an $m\times p$-matrix $B$ over $\OqQRMNC$ we have $(AB)'=B'A'$.
\end{Lem}
\begin{proof} This is an easy calculation.
\end{proof}
\begin{Prop}\label{tracechangeordering} If the $N\times N$-matrices $M$, $P$ in noncommuting variables satisfy the relation $M_{01}R_{12}P_{02}=P_{02}R_{12}M_{01}$ then $\Tr(D^{2}MP)=\Tr((D')^{2}PM)$.
\end{Prop}
\begin{proof} We write $M_{01}R_{12}P_{02}$ and $P_{02}R_{12}M_{01}$ out. In both of these expressions, we take the sum of $q^{2(N-j+i-1)}$ times the element on position $e_{ij}\otimes e_{ji}$, $i,j\in[N]$. Now the result follows.
\end{proof}
\begin{Lem} For all $k\geq 0$ we have
\begin{equation}
((XY)^{k}X)_{01}R_{12}Y_{02}=Y_{02}R_{12}((XY)^{k}X)_{01}.\label{FRTforhigherorder}
\end{equation}
\end{Lem}
\begin{proof} This follows from \cref{FRTOqQRMNC} and induction.
\end{proof}
\begin{Thm}\label{invariance} For all $N\geq 2$, the coefficients $\sigma_{q,Q}(i)$, $0\leq i\leq N$ and the $s_{q,Q}(k)$, $k\geq 0$ are invariant under $'$ (see also \cite[Remark 3.42]{DCFtfoqGLNC}):
\begin{align*}
(\sigma_{q,Q}(i))' & =\sigma_{q,Q}(i), & (s_{q,Q}(k))' & =\Tr((D')^{2}(YX)^{k})=\Tr(D^{2}(XY)^{k})=s_{q,Q}(k).
\end{align*}
\end{Thm}
\begin{proof} For the $s_{q,Q}(k)$, this follows from \cref{tracechangeordering} and \cref{FRTforhigherorder}. For the $\sigma_{q,Q}(i)$, the invariance now follows from \cref{Newton}.
\end{proof}
While \cite[Theorem 1.3]{JWtcotreavqm} (\cref{sigmaqiJW}) describes the $\sigma_{q}(i)$ in terms of the entries of the matrix $A=XY$ and the exceedance, we can analogously describe the $\sigma_{q}(i)$ using the entries of the matrix $B=YX$ and the \emph{anti-exceedance} $a(\sigma)=|\{i\in[N]\mid\sigma(i)<i\}|$ for $\sigma\in S_{N}$.
\begin{Thm} For $N\geq 2$ and $0\leq i\leq N$ we have
\begin{equation}
\sigma_{q}(i)=\sum_{J\in\binom{[N]}{i}}{\sum_{\sigma\in\Sym(J)}{q^{2(iN-\wt(J))}(-q)^{-l(\sigma)}q^{-a(\sigma)}b_{\sigma(j_{1})j_{1}}\dots b_{\sigma(j_{i})j_{i}}}}.\label{sigmaqiJWafterinvariance}
\end{equation}
\end{Thm}
\begin{proof} This follows by applying \cref{invariance} to \cref{sigmaqiJW}.
\end{proof}
\section{\CB{} for the coefficients in the quantum Cayley-Ha\-mil\-ton formula for the \rea{}}
First, we rearrange the factors in \cref{sigmaqQiGPS} in order to separate the $X_{ij}$ and $Y_{ij}$. We will use \cref{YB}, \cref{FRTOqQMNC} and \cref{FRTOqQRMNC} and we write $\hR_{i}=\hR_{i,i+1}$, as in \cite{PScrfqm}.
\begin{Lem} For $i\geq 3$ and $1\leq k\leq i-2$ we have
\begin{gather}
\hR_{i-k}X_{01}Y_{01}\hR_{1}\dots\hR_{i-1}=X_{01}Y_{01}\hR_{1}\dots\hR_{i-1}\hR_{i-k-1},\label{formulaRXY}\\
(\hR_{i-1}\dots\hR_{i-k})(X_{01}Y_{01}\hR_{1}\dots\hR_{i-1})=(X_{01}Y_{01}\hR_{1}\dots\hR_{i-2})(\hR_{i-1}\dots\hR_{i-k-1}),\label{formulaRRXY}\\
\hR_{i-1}(X_{01}Y_{01}\hR_{1}\dots\hR_{i-1})^{i-2}=(X_{01}Y_{01}\hR_{1}\dots\hR_{i-2})^{i-2}\hR_{i-1}\dots\hR_{1},\label{formulaRXYR}\\
Y_{0k}\hR_{i-1}\dots\hR_{k}X_{0k}=\hR_{i-1}\dots\hR_{k+1}X_{0,k+1}\hR_{k}Y_{0,k+1},\label{formulaYRX}\\
Y_{0,i-1}\dots Y_{01}\hR_{i-1}\dots\hR_{1}X_{01}=X_{0i}\hR_{i-1}\dots\hR_{1}Y_{0i}\dots Y_{02},\label{formulaYYRX}\\
Y_{0i}\dots Y_{01}\hR_{1}\dots\hR_{i-1}=\hR_{1}\dots\hR_{i-1}Y_{0i}\dots Y_{01}.\label{formulaYYR}
\end{gather}
\end{Lem}
\begin{proof} The left hand side of \cref{formulaRXY} equals
\begin{align*}
 & X_{01}Y_{01}\hR_{1}\dots\hR_{i-k-2}\hR_{i-k}\hR_{i-k-1}\hR_{i-k}\hR_{i-k+1}\dots\hR_{i-1}\\
={} & X_{01}Y_{01}\hR_{1}\dots\hR_{i-k-2}\hR_{i-k-1}\hR_{i-k}\hR_{i-k-1}\hR_{i-k+1}\dots\hR_{i-1}\\
={} & X_{01}Y_{01}\hR_{1}\dots\hR_{i-1}\hR_{i-k-1},
\end{align*}
where we used \cref{YB}. Now \cref{formulaRRXY} follows from \cref{formulaRXY} and \cref{formulaRXYR} follows from \cref{formulaRRXY}. The left hand side of \cref{formulaYRX} equals $\hR_{i-1}\dots\hR_{k+1}Y_{0k}\hR_{k}X_{0k}$ and \cref{formulaYRX} follows from \cref{FRTOqQRMNC}. By \cref{formulaYRX}, the left hand side of \cref{formulaYYRX} equals
\[
Y_{0,i-1}\hR_{i-1}X_{0,i-1}\hR_{i-2}\dots\hR_{1}Y_{0,i-1}\dots Y_{02},
\]
which equals the right hand side of \cref{formulaYYRX} by \cref{FRTOqQRMNC}. For $i=2$, \cref{formulaYYR} follows from \cref{FRTOqQRMNC}. By induction, the left hand side of \cref{formulaYYR} equals
\[
Y_{0i}\hR_{1}\dots\hR_{i-2}Y_{0,i-1}\dots Y_{01}\hR_{i-1}=\hR_{1}\dots\hR_{i-2}\underbrace{Y_{0i}Y_{0,i-1}\hR_{i-1}}_{=\hR_{i-1}Y_{0i}Y_{0,i-1}}Y_{0,i-2}\dots Y_{01},
\]
which equals the right hand side of \cref{formulaYYR}.
\end{proof}
\begin{Lem} For $i\geq 3$ and $2\leq k\leq i-1$ we have
\begin{gather}
\hR_{i-1}\dots\hR_{i-k+1}\hR_{i-k}\dots\hR_{i-1}=\hR_{i-k}\dots\hR_{i-1}\hR_{i-2}\dots\hR_{i-k},\label{formulaRRRR}\\
(\hR_{1}\dots\hR_{i-1})^{k}=(\hR_{1}\dots\hR_{i-2})^{k}\hR_{i-1}\dots\hR_{i-k},\label{formulaRR}\\
(\hR_{1}\dots\hR_{i-2})^{i-1}\hR_{i-1}\dots\hR_{1}=(\hR_{1}\dots\hR_{i-1})^{i-1}.\label{formulaRRR}
\end{gather}
\end{Lem}
\begin{proof} For $k=2$, \cref{formulaRRRR} follows from \cref{YB}. Now we use induction. The left hand side of \cref{formulaRRRR} equals
\begin{align*}
 & \hR_{i-1}\dots\hR_{i-k+2}\hR_{i-k+1}\hR_{i-k}\hR_{i-k+1}\dots\hR_{i-1}\\
={} & \hR_{i-1}\dots\hR_{i-k+2}\hR_{i-k}\hR_{i-k+1}\hR_{i-k}\hR_{i-k+2}\dots\hR_{i-1}\\
={} & \hR_{i-k}\hR_{i-1}\dots\hR_{i-k+2}\hR_{i-k+1}\dots\hR_{i-1}\hR_{i-k}\\
={} & \hR_{i-k}\hR_{i-k+1}\dots\hR_{i-1}\hR_{i-2}\dots\hR_{i-k+1}\hR_{i-k},
\end{align*}
where the last equality follows by induction. This equals the right hand side of \cref{formulaRRRR}. For $k=2$, \cref{formulaRR} follows from
\[
(\hR_{1}\dots\hR_{i-1})(\hR_{1}\dots\hR_{i-1})=\hR_{1}\dots\hR_{i-2}\hR_{1}\dots\hR_{i-3}\underbrace{\hR_{i-1}\hR_{i-2}\hR_{i-1}}_{=\hR_{i-2}\hR_{i-1}\hR_{i-2}},
\]
which equals the right hand side of \cref{formulaRR}. By induction, the left hand side of \cref{formulaRR} equals
\begin{align*}
 & (\hR_{1}\dots\hR_{i-2})^{k-1}\hR_{i-1}\dots\hR_{i-k+1}\hR_{1}\dots\hR_{i-1}\\
={} & (\hR_{1}\dots\hR_{i-2})^{k-1}\hR_{1}\dots\hR_{i-k-1}\hR_{i-1}\dots\hR_{i-k+1}\hR_{i-k}\dots\hR_{i-1},
\end{align*}
which equals the right hand side of \cref{formulaRR} by \cref{formulaRRRR}. Now \cref{formulaRRR} follows from \cref{formulaRR} by taking $k=i-1$.
\end{proof}
Now we calculate the middle part in \cref{sigmaqQiGPS} in order to have $X$ in front and $Y$ at the back.
\begin{Lem} For $i\geq 2$ we have
\begin{equation}
(X_{01}Y_{01}\hR_{1}\dots\hR_{i-1})^{i}=(X_{01}\dots X_{0i})(\hR_{1}\dots\hR_{i-1})^{i}(Y_{0i}\dots Y_{01}).\label{middlepartsigmaqQi}
\end{equation}
\end{Lem}
\begin{proof} For $i=2$ we have
\[
X_{01}Y_{01}\hR_{1}X_{01}Y_{01}\hR_{1}=X_{01}X_{02}\hR_{1}Y_{02}Y_{01}\hR_{1}=X_{01}X_{02}\hR_{1}^{2}Y_{02}Y_{01}.
\]
By \cref{formulaRXYR}, induction, \cref{formulaYYRX} and \cref{formulaYYR}, the left hand side of \cref{middlepartsigmaqQi} equals
\begin{align*}
 & (X_{01}Y_{01}\hR_{1}\dots\hR_{i-2})(X_{01}Y_{01}\hR_{1}\dots\hR_{i-2})^{i-2}\hR_{i-1}\dots\hR_{1}X_{01}Y_{01}\hR_{1}\dots\hR_{i-1}\\
={} & (X_{01}\dots X_{0,i-1})(\hR_{1}\dots\hR_{i-2})^{i-1}Y_{0,i-1}\dots Y_{01}\hR_{i-1}\dots\hR_{1}X_{01}Y_{01}\hR_{1}\dots\hR_{i-1}\\
={} & (X_{01}\dots X_{0,i-1})(\hR_{1}\dots\hR_{i-2})^{i-1}X_{0i}\hR_{i-1}\dots\hR_{1}Y_{0i}\dots Y_{01}\hR_{1}\dots\hR_{i-1}\\
={} & (X_{01}\dots X_{0i})(\hR_{1}\dots\hR_{i-2})^{i-1}\hR_{i-1}\dots\hR_{1}\hR_{1}\dots\hR_{i-1}Y_{0i}\dots Y_{01},
\end{align*}
which equals the right hand side of \cref{middlepartsigmaqQi} by \cref{formulaRRR}.
\end{proof}
\begin{Prop} For $N\geq 2$ and $0\leq i\leq N$ we have
\begin{equation}
\sigma_{q,Q}(i)=q^{(i-1)i}\alpha_{i}v_{1\dots N}X_{01}\dots X_{0i}Y_{0i}\dots Y_{01}u_{1\dots N}.\label{sigmaqQiseparated}
\end{equation}
\end{Prop}
\begin{proof} For $i=0$, \cref{sigmaqQiseparated} gives $\alpha_{0}\|u_{1\dots N}\|^{2}$, which gives the correct value $1$ by \cref{alphai} and for $i=1$, \cref{sigmaqQiseparated} equals \cref{sigmaqQiGPS}. Now let $2\leq i\leq N$. Then \cref{sigmaqQiseparated} follows from \cref{formulaYYR}, \cref{sigmaqQiGPS}, \cref{middlepartsigmaqQi} and \cref{RLC}.
\end{proof}
In order to formulate our main results, we need the \mpq{} minors defined in \cref{mpqminor}. These appeared before in \cite[12]{Sasafetifqmttmc}. Here, $(-q)_{\sigma,r}$ and $(-q)_{\sigma,c}$ are defined in \cref{minusqsigmarandminusqsigmac}.
\begin{Prop}\label{mpqminor} For $J,K\in\binom{[N]}{i}$, we have
\begin{align*}
 & q^{-2l(\sigma)}\sum_{\substack{\tau\in K^{[i]}\\\text{injective}}}{(-q)_{\sigma,r}(-q)_{\tau,c}X_{\sigma(1)\tau(1)}\dots X_{\sigma(i)\tau(i)}} & & \text{for }\sigma\in J^{[i]}\text{ as below}\\*
={} & q^{-2l(\tau)}\sum_{\substack{\sigma\in J^{[i]}\\\text{injective}}}{(-q)_{\sigma,r}(-q)_{\tau,c}X_{\sigma(1)\tau(1)}\dots X_{\sigma(i)\tau(i)}} & & \text{for }\tau\in K^{[i]}\text{ as below},
\end{align*}
where $\sigma$ and $\tau$ are both injective, in which case we call this element the \emph{\mpq{} minor} of $X$ and write it as $[X]_{q,Q,J,K}$, or $\sigma$ and $\tau$ are both not injective, in which case this element is $0$.
\end{Prop}
\begin{proof} This can be proven analogously as \cite[Lemma 4.1.1]{PWqlg}.
\end{proof}
Note that $[X]_{q,Q,\emptyset,\emptyset}=1$. By choosing $\sigma(n)=j_{n}$ respectively $\tau(n)=k_{n}$, $n\in[i]$, we obtain in particular
\[
[X]_{q,Q,J,K}=\sum_{\tau\in S_{i}}{(-q)_{\tau,c}X_{j_{1}k_{\tau(1)}}\dots X_{j_{i}k_{\tau(i)}}}=\sum_{\sigma\in S_{i}}{(-q)_{\sigma,r}X_{j_{\sigma(1)}k_{1}}\dots X_{j_{\sigma(i)}k_{i}}}.
\]
Since $Y=X^{*}$, it is natural to define the \mpq{} minor of $Y$ by $[Y]_{q,Q,J,K}=[X]_{q,Q,K,J}^{*}$. The \emph{\mpq{} determinants} are given by
\begin{align*}
\detqQ(X) & =[X]_{q,Q,[N],[N]}, & \detqQ(Y) & =[Y]_{q,Q,[N],[N]}.
\end{align*}
Note that $DX$ and $XD'$ satisfy the same relations \cref{relationsXX} as $X$. Similarly, $YD$ and $D'Y$ satisfy the same relations \cref{FRTOqQRMNC} as $Y$. Hence it is meaningful to consider the \mpq{} minors of these 4 matrices. We verify that
\begin{equation}
\begin{aligned}
[DX]_{q,Q,J,K} & =q^{\wt(J-1)}[X]_{q,Q,J,K}, & [XD']_{q,Q,J,K} & =q^{\wt(K'-1)}[X]_{q,Q,J,K},\\*
[YD]_{q,Q,J,K} & =q^{\wt(K-1)}[Y]_{q,Q,J,K}, & [D'Y]_{q,Q,J,K} & =q^{\wt(J'-1)}[Y]_{q,Q,J,K}.
\end{aligned}\label{mpqminorofXandYwithDandDprime}
\end{equation}
Here, $\wt(J-1)=\wt(J)-i$ and $\wt(J'-1)=iN-\wt(J)$. For $J\in\binom{[N]}{i}$ and $\sigma\in J^{[i]}$ we verify that
\begin{align*}
((-q)_{\sigma,r})' & =(-q)_{\sigma',c}, & ((-q)_{\sigma,c})' & =(-q)_{\sigma',r},
\end{align*}
where $\sigma':[i]\to J':n\mapsto(\sigma(i-n+1))'$.
\begin{Lem} For $J,K\in\binom{[N]}{i}$, we have
\begin{align}
[X]_{q,Q,J,K}' & =[X]_{q,Q,K',J'}, & [Y]_{q,Q,J,K}' & =[Y]_{q,Q,K',J'},\nonumber\\*
[DX]_{q,Q,J,K}' & =[XD']_{q,Q,K',J'}, & [YD]_{q,Q,J,K}' & =[D'Y]_{q,Q,K',J'}.\label{primeappliedtompqminor}
\end{align}
Note the analogy with \cref{definitionprime}. In particular (see also \cite[Lemma 4.2.3]{PWqlg}) $(\detqQ(X))'=\detqQ(X)$.
\end{Lem}
\begin{proof} This can be proven analogously as \cite[Lemma 4.3.1]{PWqlg}.
\end{proof}
Now we prove \cref{theoremsigmaqQiCB}.
\begin{proof}[Proof of \cref{theoremsigmaqQiCB}] In \cref{sigmaqQiseparated}, we write $u_{1\dots N}$ out using \cref{LC} and we denote by $\sigma$ respectively $\tau$ the permutation in $v_{1\dots N}$ respectively $u_{1\dots N}$. Then the corresponding term in \cref{sigmaqQiseparated} is $0$ unless
\begin{equation}
\sigma(i+1)=\tau(i+1),\dots,\sigma(N)=\tau(N).\label{sigmaequalstauoniplusonetillN}
\end{equation}
Hence, we can write \cref{sigmaqQiseparated} as
\begin{align*}
\sigma_{q,Q}(i)={} & q^{(i-1)i}\alpha_{i}\sum_{\mathclap{\substack{\sigma,\tau\in S_{N},\cref{sigmaequalstauoniplusonetillN},\\\nu\in[N]^{[i]}}}}{(-q)_{\sigma,r}(-q)_{\tau,c}X_{\sigma(1),\nu(1)}\dots X_{\sigma(i),\nu(i)}Y_{\nu(i),\tau(i)}\dots Y_{\nu(1),\tau(1)}}\\
={} & q^{(i-1)i}\alpha_{i}\sum_{J\in\binom{[N]}{i}}{\sum_{\substack{\tilde{\sigma},\tilde{\tau}\in J^{[i]}\\\text{injective}}}{\sum_{\substack{\xi\in([N]\setminus J)^{[N]\setminus[i]}\\\text{injective}}}{\sum_{\nu\in[N]^{[i]}}}}}\\*
 & (-q)_{\sigma,r}(-q)_{\tau,c}X_{\tilde{\sigma}(1),\nu(1)}\dots X_{\tilde{\sigma}(i),\nu(i)}Y_{\nu(i),\tilde{\tau}(i)}\dots Y_{\nu(1),\tilde{\tau}(1)},
\end{align*}
where $\sigma|_{[i]}=\tilde{\sigma}$, $\tau|_{[i]}=\tilde{\tau}$ and $\sigma|_{[N]\setminus[i]}=\tau|_{[N]\setminus[i]}=\xi$. We have that
\begin{align*}
(-q)_{\sigma,r} & =(-q)_{\tilde{\sigma},r}(-q)_{\xi,r}\prod_{\mathclap{\substack{m\in J,n\in[N]\setminus J,\\m>n}}}{(-q_{mn})}, & (-q)_{\tau,c} & =(-q)_{\tilde{\tau},c}(-q)_{\xi,c}\prod_{\mathclap{\substack{m\in J,n\in[N]\setminus J,\\m>n}}}{(-q_{nm})}.
\end{align*}
Here, $\prod_{m\in J,n\in[N]\setminus J,m>n}{(-q_{mn})(-q_{nm})}=q^{2\left(\wt(J)-\frac{i(i+1)}{2}\right)}$. The $(-q)_{\tilde{\sigma},r}$ and $(-q)_{\tilde{\tau},c}$ can be used to apply \cref{mpqminor}. Then we only need the $\nu$ which are injective. Then $\sum_{\nu\in[N]^{[i]}}$ is the same as $\sum_{K\in\binom{[N]}{i}}\sum_{\nu\in K^{[i]}\text{ injective}}$ and we obtain $[X]_{q,Q,J,K}$ and $[Y]_{q,Q,K,J}$. The remaining factors are $\sum_{\xi\in S_{N-i}}{q^{2l(\xi)}}=[N-i]!$ and $\sum_{\nu\in S_{i}}{q^{2l(\nu)}}=[i]!$, where we used \cref{normLC}. These cancel with $\alpha_{i}$. Using \cref{invariance} and \cref{primeappliedtompqminor}, we obtain the second formula.
\end{proof}
In particular we have
\begin{equation}
\sigma_{q,Q}(N)=q^{(N-1)N}\detqQ(X)\detqQ(Y)=q^{(N-1)N}\detqQ(Y)\detqQ(X).\label{sigmaqQN}
\end{equation}
For $T$, \cref{sigmaqQN} reduces to \cref{sigmaqTN}. Note that by using \cref{mpqminorofXandYwithDandDprime}, we have in \cref{theoremsigmaqQiCB} the factors $q^{2(\wt(J)-i)}$ and $q^{2(iN-\wt(J))}$ for respectively $X$ followed by $Y$ or vice versa. These factors also occur in \cref{sigmaqiJW} and \cref{sigmaqiJWafterinvariance}, where $A=XY$ respectively $B=YX$ and again $X$ is followed by $Y$ or vice versa.
\begin{Rem} The coefficient $\sigma_{q,Q}(N-1)$ can also be found as follows. We multiply \cref{CH} to the left with $D^{2}(XY)^{-1}$, take the trace and use \cref{Newton} to obtain
\[
\sigma_{q,Q}(N-1)=\Tr(((D')^{2}XY)^{-1})\sigma_{q,Q}(N).
\]
Now, $(XY)^{-1}$ can e.g.~be found using \cref{theoremtINplusXYinverse} for $t=0$. This results in the same formula for $\sigma_{q,Q}(N-1)$ as in \cref{theoremsigmaqQiCB}.
\end{Rem}
\begin{Rem} In \cite{MerqdBCf}, a \CB{} formula is proved for the product $UV$, where the entries of $U$ and $V$ (not necessarily square) both satisfy the relations \cref{relationsXX} in $\OqMNC$ and with $[U_{ij},V_{kl}]=0$. In the setting of \cref{theoremsigmaqQiCB}, this last relation does not hold since $X_{ij}$ does in general not commute with $Y_{kl}$.
\end{Rem}
\begin{Rem} In general, the $s_{q}(i)$ in e.g.~the ordering used in the second column of \cref{tableTtwothree} \emph{do not} satisfy \cref{observation}, e.g.~for $T$ and $N=2$, $s_{q}(2)$ then contains the term $-q^{3}(q-q^{-1})T_{11}^{2}T_{22}^{2}$.
\end{Rem}
\section{\CB{} for \texorpdfstring{$(tI_{N}+XY)^{-1}$}{the inverse of tIN+XY}}
Our original motivation is the observation that \cite[(4.14)]{WorCstaragbue} satisfies \cref{observation}. More precisely, we have
\begin{equation}
(I_{2}+T^{*}T)^{-1}=\frac{1}{\mathcal{C}_{2}}\begin{pmatrix*}[c]1+q^{2}T_{22}^{2}+T_{12}T_{12}^{*} & -T_{11}T_{12}\\-T_{12}^{*}T_{11} & 1+q^{2}T_{11}^{2}\end{pmatrix*}\label{ItwoplusTstarTinverse}
\end{equation}
where $\mathcal{C}_{2}=1+q^{2}T_{11}^{2}+T_{22}^{2}+T_{12}^{*}T_{12}+q^{2}T_{11}^{2}T_{22}^{2}$ is a self-adjoint central element. In this Section, we will prove \cref{theoremtINplusXYinverse}, where such a formula for all $N\geq 2$ in the \mpq{} case with $X$ and $Y$ is given. We will need the \emph{\mpq{} Laplace expansions} in \cref{Laplexp}. By applying $^{*}$, we obtain \mpq{} Laplace expansions for minors of $Y$.
\begin{Prop}\label{Laplexp} For $J,L\in\binom{[N]}{i}$ and $j,l\in[N]$ we have
\begin{align}
\delta_{jl}[X]_{q,Q,J,L} & =(-q)_{J<l}^{-1}\sum_{k\in L}{(-q)_{L<k}X_{jk}[X]_{q,Q,J\setminus\{l\},L\setminus\{k\}}} & & \text{if }j,l\in J\label{LaplexprowXminorX}\\
 & =(-q)_{l>L}^{-1}\sum_{k\in J}{(-q)_{k>J}X_{kj}[X]_{q,Q,J\setminus\{k\},L\setminus\{l\}}} & & \text{if }j,l\in L\nonumber\\
 & =(-q)_{j<J}^{-1}\sum_{k\in L}{(-q)_{k<L}[X]_{q,Q,J\setminus\{j\},L\setminus\{k\}}X_{lk}} & & \text{if }j,l\in J\label{LaplexpminorXrowX}\\
 & =(-q)_{L>j}^{-1}\sum_{k\in J}{(-q)_{J>k}[X]_{q,Q,J\setminus\{k\},L\setminus\{j\}}X_{kl}} & & \text{if }j,l\in L.\label{LaplexpminorXcolumnX}
\end{align}
\end{Prop}
\begin{proof} This can be proven analogously as \cite[Corollary 4.4.4]{PWqlg}.
\end{proof}
We will need to be able to rewrite $X_{ij}Y_{kl}$ as an expression where all $Y_{ij}$ are placed before $X_{ij}$. For now, \cref{relationsXY} is not sufficient because of the sum $\sum_{m=1}^{j-1}{X_{im}Y_{ml}}$.
\begin{Lem}\label{relationsumofXYtemporary} For $1\leq\alpha\leq j-1$ we have
\begin{align*}
\sum_{m=1}^{j-1}{X_{im}Y_{ml}}={} & \sum_{m=j-\alpha}^{j-1}{q^{2(j-m-1)}\left(q_{il}Y_{ml}X_{im}+\delta_{il}(1-q^{2})\sum_{k=i+1}^{N}{Y_{mk}X_{km}}\right)}\\*
 & +q^{2\alpha}\sum_{m=1}^{j-1-\alpha}{X_{im}Y_{ml}}.
\end{align*}
\end{Lem}
\begin{proof} This follows by induction on $\alpha$.
\end{proof}
\begin{Lem}\label{relationsumofXY} We have
\[
\sum_{m=1}^{j-1}{X_{im}Y_{ml}}=\sum_{m=1}^{j-1}{q^{2(j-m-1)}\left(q_{il}Y_{ml}X_{im}+\delta_{il}(1-q^{2})\sum_{\mathclap{k=i+1}}^{N}{Y_{mk}X_{km}}\right)}.
\]
\end{Lem}
\begin{proof} This follows from \cref{relationsumofXYtemporary} for $\alpha=j-1$.
\end{proof}
Using \cref{relationsumofXY}, we can rewrite \cref{relationsXY} and then we are able to rewrite $X_{ij}Y_{kl}$ as an expression where all $Y_{ij}$ are placed before $X_{ij}$. Similarly, we will need to be able to rewrite $[X]_{q,Q,J,K}Y_{\beta\alpha}$ as an expression where all $Y_{ij}$ are placed before minors of $X$. First, we prove a formula between minors of $X$ and $Y_{ij}$ in \cref{relationsminorXtimesY}. Note that this is a generalization of \cref{relationsXY}.
\begin{Prop}\label{relationsminorXtimesY} For $J,K\in\binom{[N]}{i}$ we have
\begin{align*}
 & q_{K\beta}[X]_{q,Q,J,K}Y_{\beta\alpha}\\*
 & +\delta_{\beta\in K}(1-q^{2})q_{K\beta}(-q)_{K>\beta}^{-1}\sum_{\substack{m=1\\m\notin K}}^{\beta-1}{(-q)_{(K\setminus\{\beta\})>m}[X]_{q,Q,J,(K\cup\{m\})\setminus\{\beta\}}Y_{m\alpha}}\\
={} & q_{J\alpha}Y_{\beta\alpha}[X]_{q,Q,J,K}\\*
 & +\delta_{\alpha\in J}(1-q^{2})q_{J\alpha}(-q)_{J<\alpha}^{-1}\sum_{\substack{l=\alpha+1\\l\notin J}}^{N}{(-q)_{(J\setminus\{\alpha\})<l}Y_{\beta l}[X]_{q,Q,(J\cup\{l\})\setminus\{\alpha\},K}}.
\end{align*}
\end{Prop}
\begin{proof} We calculate $q_{K\beta}[X]_{q,Q,J,K}Y_{\beta\alpha}$. If $\alpha\notin J$, $\beta\notin K$ then this follows from \cref{relationsXY}. If $\alpha\in J$, $\beta\notin K$ then we use \cref{LaplexprowXminorX} with $j=l=\alpha$. For $[X]_{q,Q,J\setminus\{\alpha\},K\setminus\{k\}}Y_{\beta\alpha}$ we use the previous case. Then we have $X_{\alpha k}Y_{\beta\alpha}$. Here we use \cref{relationsXY}. Then we use \cref{LaplexprowXminorX} again. The case when $\alpha\notin J$, $\beta\in K$ can be proven analogously by using \cref{LaplexpminorXcolumnX} with $j=l=\beta$. If $\alpha\in J$, $\beta\in K$ then we use induction on $i$. For $i=1$, the result is \cref{relationsXY}. Now suppose $i\geq 2$. Then there exists $\alpha'\in J\setminus\{\alpha\}$ and we use \cref{LaplexpminorXrowX} with $j=l=\alpha'$. Then we have $\sum_{k\in K}$. On $X_{\alpha'k}Y_{\beta\alpha}$ we apply \cref{relationsXY}. For $k\neq\beta$ we apply induction on $[X]_{q,Q,J\setminus\{\alpha'\},K\setminus\{k\}}Y_{\beta\alpha}$. For $k=\beta$ we can use a previous case. Now we can group some terms using \cref{Laplexp} and we obtain all terms in \cref{relationsminorXtimesY}. The remaining terms give
\begin{align*}
 & -(1-q^{2})(-q)_{\alpha'<J}^{-1}[X]_{q,Q,J\setminus\{\alpha'\},K\setminus\{\beta\}}\Bigg(q_{K\beta}(-q)_{\beta<K}\sum_{m=1}^{\beta-1}{X_{\alpha'm}Y_{m\alpha}}\\
 & -\sum_{m=1}^{\beta-1}{q^{2|\{n\in K\setminus\{\beta\}\mid m<n\}|}q_{K\beta}(-q)_{K>\beta}^{-1}X_{\alpha'm}Y_{m\alpha}}\\
 & -(1-q^{2})\sum_{\substack{m=1\\m\in K}}^{\beta-1}{q^{2|\{n\in K\setminus\{\beta\}\mid m<n\}|}q_{K\beta}(-q)_{K>\beta}^{-1}\sum_{n=1}^{m-1}{X_{\alpha'n}Y_{n\alpha}}}\Bigg).
\end{align*}
Let $n_{m}=|\{n\in K\mid m<n<\beta\}|$. Then the 3 terms between brackets give
\[
q_{K\beta}(-q)_{\beta<K}\left(\sum_{m=1}^{\beta-1}{(1-q^{2n_{m}})X_{\alpha'm}Y_{m\alpha}}-(1-q^{2})\sum_{\substack{m=1\\m\in K}}^{\beta-1}{q^{2n_{m}}\sum_{n=1}^{m-1}{X_{\alpha'n}Y_{n\alpha}}}\right).
\]
If $\{n\in K\mid 1\leq n<\beta\}=\emptyset$ then this gives $0$. Otherwise we denote by $k_{1}<\dots<k_{t}$ the elements of this set. Now we write the above sums in terms of these $k_{j}$. In this case we also obtain $0$.
\end{proof}
For now, \cref{relationsminorXtimesY} is not sufficient to have all $Y_{ij}$ placed before minors of $X$ because of the sum $\sum_{m=1,m\notin K}^{\beta-1}$.
\begin{Lem}\label{relationsumofminorXtimesYtemporary} Let $J,K\in\binom{[N]}{i}$ and $\beta\in K$. Let $m_{1}<\dots<m_{\mu}$ be such that $\{m_{1},\dots,m_{\mu}\}=\{1,\dots,\beta-1\}\setminus K$. Then for $1\leq\gamma\leq\mu$ we have
\begin{align*}
 & \sum_{r=1}^{\gamma}{(-q)_{(K\setminus\{\beta\})>m_{r}}[X]_{q,Q,J,(K\cup\{m_{r}\})\setminus\{\beta\}}Y_{m_{r}\alpha}}\\
={} & \sum_{r=1}^{\gamma}{q^{2(\gamma-r)}(-q)_{(K\setminus\{\beta\})>m_{r}}q_{(K\cup\{m_{r}\})\setminus\{\beta\},m_{r}}^{-1}\Bigg(q_{J\alpha}Y_{m_{r}\alpha}[X]_{q,Q,J,(K\cup\{m_{r}\})\setminus\{\beta\}}}\\*
 & +\delta_{\alpha\in J}(1-q^{2})q_{J\alpha}(-q)_{J<\alpha}^{-1}\\*
 & \hphantom{+\delta_{\alpha\in J}(1-q^{2})}\cdot\bigg(\sum_{\substack{l=\alpha+1\\l\notin J}}^{N}{(-q)_{(J\setminus\{\alpha\})<l}Y_{m_{r}l}[X]_{q,Q,(J\cup\{l\})\setminus\{\alpha\},(K\cup\{m_{r}\})\setminus\{\beta\}}}\bigg)\Bigg).
\end{align*}
\end{Lem}
\begin{proof} This follows by induction on $\gamma$.
\end{proof}
\begin{Lem}\label{relationsumofminorXtimesY} For $J,K\in\binom{[N]}{i}$ and $\beta\in K$ we have
\begin{align*}
 & \sum_{\substack{m=1\\m\notin K}}^{\beta-1}{(-q)_{(K\setminus\{\beta\})>m}[X]_{q,Q,J,(K\cup\{m\})\setminus\{\beta\}}Y_{m\alpha}}\\
={} & \sum_{\substack{m=1\\m\notin K}}^{\beta-1}{q^{2(([N]\setminus K)_{<\beta}^{m}-1)}(-q)_{(K\setminus\{\beta\})>m}q_{K\setminus\{\beta\},m}^{-1}\Bigg(q_{J\alpha}Y_{m\alpha}[X]_{q,Q,J,(K\cup\{m\})\setminus\{\beta\}}}\\*
 & +\delta_{\alpha\in J}(1-q^{2})q_{J\alpha}(-q)_{J<\alpha}^{-1}\\*
 & \hphantom{+\delta_{\alpha\in J}(1-q^{2})}\cdot\bigg(\sum_{\substack{l=\alpha+1\\l\notin J}}^{N}{(-q)_{(J\setminus\{\alpha\})<l}Y_{ml}[X]_{q,Q,(J\cup\{l\})\setminus\{\alpha\},(K\cup\{m\})\setminus\{\beta\}}}\bigg)\Bigg),
\end{align*}
where $([N]\setminus K)_{<\beta}^{m}=n$ is such that $m$ is the $n$-th element of $\{1,\dots,\beta-1\}\setminus K$, ordered from large to small.
\end{Lem}
\begin{proof} This follows from \cref{relationsumofminorXtimesYtemporary} for $\gamma=\mu$.
\end{proof}
Note that \cref{relationsumofminorXtimesY} is a generalization of \cref{relationsumofXY}. Using \cref{relationsumofminorXtimesY}, we can rewrite \cref{relationsminorXtimesY} and then we are able to rewrite $[X]_{q,Q,J,K}Y_{\beta\alpha}$ as an expression where all $Y_{ij}$ are placed before minors of $X$.

Now we prove \cref{theoremtINplusXYinverse}.
\begin{proof}[Proof of \cref{theoremtINplusXYinverse}] We calculate the entry $(i,k)$ of the product of the above expression, without $\mathcal{C}_{N}$, with $tI_{N}+XY$. This gives a sum over $j\in[N]$ of the product of entry $(i,j)$ of the above expression, with $(tI_{N}+XY)_{jk}$. We separate $j\neq i$ and $j=i$. Now we consider the part of degree $g$ in $X$ and $Y$, $0\leq g\leq N$ and we separate the cases $i\neq k$ and $i=k$. For $g=0,1,N$, the formula is easily verified. For $2\leq g\leq N-1$, we first consider the case $i\neq k$. Here, $\sum_{j=1,j\neq i}^{N}\sum_{K\in\binom{[N]}{g},i,j\in K}$ is the same as $\sum_{K\in\binom{[N]}{g},i\in K}\sum_{j\in K,j\neq i}$. Now we apply \cref{LaplexpminorXcolumnX} on
\[
\sum_{j\in K}{(-q)_{K>j}[X]_{q,Q,K\setminus\{j\},L}X_{jm}}.
\]
If $m\in L$ then there exists $\alpha\in[N]\setminus L$ and we rewrite $L=(L\cup\{\alpha\})\setminus\{\alpha\}$. Hence this gives $0$ by \cref{LaplexpminorXcolumnX}. If $m\notin L$ this gives $[X]_{q,Q,K,L\cup\{m\}}(-q)_{L>m}$ by \cref{LaplexpminorXcolumnX}. On $[X]_{q,Q,K,L\cup\{m\}}Y_{mk}$ we apply \cref{relationsminorXtimesY} combined with \cref{relationsumofminorXtimesY}. This results in
\begin{align}
 & q_{Lm}^{-1}\Bigg(q_{Kk}Y_{mk}[X]_{q,Q,K,L\cup\{m\}}+\delta_{k\in K}(1-q^{2})\sum_{\substack{l=k+1\\l\notin K}}^{N}{\dots}\label{theoremtINplusXYinverseprooffirstline}\\*
 & -(1-q^{2})\sum_{\substack{\mu=1\\\mu\notin L}}^{m-1}{\dots\bigg(q_{Kk}Y_{\mu k}[X]_{q,Q,K,L\cup\{\mu\}}+\delta_{k\in K}(1-q^{2})\sum_{\substack{l=k+1\\l\notin K}}^{N}{\dots}\bigg)}\Bigg).\label{theoremtINplusXYinverseproofsecondline}
\end{align}
In the first part of \cref{theoremtINplusXYinverseproofsecondline} we have
\[
-(1-q^{2})\sum_{m\notin L}{\sum_{\substack{\mu=1\\\mu\notin L}}^{m-1}{q^{2(([N]\setminus L)_{<m}^{\mu}-1)}(-q)_{L>\mu}q_{L\mu}^{-1}Y_{\mu k}[X]_{q,Q,K,L\cup\{\mu\}}}}.
\]
We rewrite this as $\sum_{\mu\notin L}\sum_{m=\mu+1,m\notin L}^{N}$. Let $[N]\setminus L=\{l_{1},\dots,l_{\lambda}\}$ with $l_{1}<\dots<l_{\lambda}$. Let $\mu=l_{\alpha}$. Then we have
\[
-(1-q^{2})\sum_{\beta=\alpha+1}^{\lambda}{q^{2((\{l_{1},\dots,l_{\lambda}\})_{<l_{\beta}}^{l_{\alpha}}-1)}}=-(1-q^{2})\sum_{\beta=\alpha+1}^{\lambda}{q^{2(\beta-\alpha-1)}}=q^{2(\lambda-\alpha)}-1.
\]
Now, the term with the $-1$ cancels with the first part of \cref{theoremtINplusXYinverseprooffirstline}. In what remains, we rewrite $\sum_{L\in\binom{[N]}{g-1}}\sum_{\mu\notin L}$ as $\sum_{L\in\binom{[N]}{g}}\sum_{\mu\in L}$ and we replace $L$ by $L\setminus\{\mu\}$. Now, $\lambda-\alpha=|([N]\setminus(L\setminus\{\mu\}))_{>\mu}|=|[N]_{>\mu}|-|L_{>\mu}|=N-\mu-g+|L_{<\mu}|+1$. Then we obtain
\[
\sum_{\mu\in L}{(-q)_{\mu>L}[Y]_{q,Q,L\setminus\{\mu\},K\setminus\{i\}}Y_{\mu k}},
\]
on which we apply \cref{Laplexp}. If $k\in K$ this gives $0$. If $k\notin K$ this gives $(-q)_{k>(K\cup\{k\})\setminus\{i\}}[Y]_{q,Q,L,(K\cup\{k\})\setminus\{i\}}$. We rewrite $\sum_{K\in\binom{[N]}{g},i\in K,k\notin K}$ as $\sum_{K\in\binom{[N]}{g+1},i,k\in K}$ and we replace $K$ by $K\setminus\{k\}$.

The second parts of \cref{theoremtINplusXYinverseproofsecondline} and \cref{theoremtINplusXYinverseprooffirstline} (which contain $\delta_{k\in K}$) can be treated analogously. Here, we rewrite
\[
\sum_{K\in\binom{[N]}{g},i,k\in K}\sum_{\substack{l=k+1\\l\notin K}}^{N}\text{ as }\sum_{K\in\binom{[N]}{g+1},i,k\in K,(K\setminus\{i\})_{>k}\neq\emptyset}\sum_{\substack{l=k+1\\l\neq i,l\in K}}^{N}
\]
and we replace $K$ by $K\setminus\{l\}$. We simplify the part containing $l$ and what remains in $l$ is
\begin{align*}
 & (1-q^{2})\sum_{\substack{l=k+1\\l\neq i,l\in K}}^{N}{q^{2|(K\setminus\{i\})_{<l}|}}=(1-q^{2})\sum_{\mathclap{l\in(K\setminus\{i\})_{>k}}}{q^{2|(K\setminus\{i\})_{\leq k}|+2|((K\setminus\{i\})_{>k})_{<l}|}}\\*
={} & q^{2|(K\setminus\{i\})_{\leq k}|}(1-q^{2|(K\setminus\{i\})_{>k}|})=q^{2|(K\setminus\{i\})_{\leq k}|}-q^{2|K\setminus\{i\}|}.
\end{align*}
Now we combine all terms in this calculation for $i\neq k$ and consider the cases $(K\setminus\{i\})_{>k}=\emptyset$ and $(K\setminus\{i\})_{>k}\neq\emptyset$. In both cases, this total expression gives $0$. The previous calculation for $i\neq k$ can be adapted to the case $i=k$, which results in $\sigma_{q,Q}(g)t^{N-g}$. Now we proved that our expression is a one-sided inverse. This is sufficient because $tI_{N}+XY$ is invertible since by \cref{CH}, there exists a \CH{} theorem for $tI_{N}+XY$ as well, from which we can obtain a formula for $(tI_{N}+XY)^{-1}$.
\end{proof}
\begin{Rem} By setting $t=0$, \cref{theoremtINplusXYinverse} gives $Y^{-1}X^{-1}$. By applying $'$ and using \cref{primeappliedtompqminor}, we obtain $(tI_{N}+YX)^{-1}$.
\end{Rem}
\nocite{*}
\bibliographystyle{alpha}
\bibliography{References}
\end{document}